\documentclass[a4paper,12pt]{amsart}

\usepackage[headings]{fullpage}

\calclayout

\makeindex

\makeatother

\usepackage{amsfonts,graphics,amsmath,amsthm,amsfonts,amscd,amssymb,amsmath,latexsym,euscript, enumerate}
\usepackage{epsfig}
\usepackage{xcolor}
\usepackage{flafter}
\usepackage[all,cmtip,line]{xy}
\usepackage{array}
\usepackage[english]{babel}
\usepackage{overpic}
\usepackage{subfig}
\usepackage{multirow}
\usepackage{microtype}
\usepackage{wrapfig}
\usepackage{tabularray}
\usepackage{longtable}
\usepackage{supertabular}
\usepackage{caption}
\usepackage{tikz}
\usetikzlibrary{positioning}
\usepackage{tikz-cd}
\usepackage{float}
\allowdisplaybreaks

\usepackage{graphicx}
\usepackage{hyperref}
\hypersetup{
    colorlinks=true,    
    linkcolor=blue,          
    citecolor=blue,      
    filecolor=blue,      
    urlcolor=blue           
}
\usepackage[anythingbreaks]{breakurl}

\newtheorem{theorem}{Theorem}[section]
\newtheorem{lemma}[theorem]{Lemma}
\newtheorem{proposition}[theorem]{Proposition}

\newtheorem*{theorem*}{Theorem}

\theoremstyle{plain}
\newtheorem{corollary}[theorem]{Corollary}

\theoremstyle{definition} 
\newtheorem{definition}[theorem]{Definition}
\newtheorem{definition-lemma}[theorem]{Definition-Lemma}

\newtheorem{example}[theorem]{Example}

\newtheorem{remark}[theorem]{Remark}

\numberwithin{equation}{section}

\newcommand{\R}{\mathbb{R}}
\newcommand{\Z}{\mathbb{Z}}

\newcommand{\Q}{\mathbb{Q}}

\def\P{\mathbb{P}}

\DeclareMathOperator{\ord}{ord}

\DeclareMathOperator{\Pic}{Pic}

\def\Pic{\operatorname{Pic}}

\def\Spec{\operatorname{Spec}}

\def\mult{\operatorname{mult}}

\def\Supp{\operatorname{Supp}}

\def\lct{\operatorname{lct}}

\def\Val{\operatorname{Val}}

\usepackage{mathtools}

\DeclarePairedDelimiterX{\norm}[1]{\lVert}{\rVert}{#1}

\title[Anticanonical MMP for pklt pairs]
{Structure of the Anticanonical Minimal Model Program for Potentially klt Pairs}

\begin{document}

\author[D.~Kim]{Donghyeon Kim}
\author[D.-W.~Lee]{Dae-Won Lee}
\address[Donghyeon Kim]{Department of Mathematics, Yonsei University, 50 Yonsei-ro, Seodaemun-gu, Seoul 03722, Republic of Korea}
\email{narimial0@gmail.com, whatisthat@yonsei.ac.kr}
\address[Dae-Won Lee]{Department of Mathematics, Ewha Womans University, 52 Ewhayeodae-gil, Seodaemun-gu, Seoul 03760, Republic of Korea}
\email{daewonlee@ewha.ac.kr}

\thanks{The authors are partially supported by Samsung Science and Technology Foundation under Project Number SSTF-BA2302-03. The second author is partially supported by Basic Science Research Program through the National Research Foundation of Korea (NRF) funded by the Ministry of Education (No. RS-2023-00237440 and 2021R1A6A1A10039823).}

\date{\today}

\begin{abstract}
We provide an alternative proof of the existence of the anticanonical
minimal model program for potentially klt pairs, assuming that the anticanonical divisor admits a birational Zariski decomposition. Moreover, we show that any partial anticanonical MMP starting from a potentially klt pair and containing no flips admits, after passing to $\Q$-factorial terminal models, an upstairs factorization of \((K+\Delta)\)-negative maps.
\end{abstract}

\maketitle


\section{Introduction}\label{sect:intro}
The minimal model program (MMP) for the anticanonical divisor
$-K_X$ lies at the interface between Fano geometry and the classification of algebraic varieties whose anticanonical class is pseudoeffective. A central challenge in this direction is to identify the largest natural class of pairs for which such an MMP can be run and to describe the structure of the resulting MMP.

The notion of a \emph{potentially klt} (pklt) pair was introduced
by Choi--Park \cite{CP16}. This condition captures precisely those pairs for which, if a $-(K_X+\Delta)$-minimal model $\varphi\colon (X,\Delta)\dashrightarrow (Y,\Delta_Y)$ exists, then the pair $(Y,\Delta_Y)$ is klt.  In particular, \cite{CP16} established that the potentially non-klt locus of a pklt pair is mapped birationally onto the non-klt locus on any anticanonical minimal model, and gave a characterization of Fano type varieties.

Recently, \cite{CJKL25} showed, via a valuative approach, that the
log canonical threshold of any pklt triple can be computed by a
quasi-monomial valuation, extending the celebrated result of Xu
\cite{Xu20} from the klt setting to the pklt setting.
As a consequence, they established the existence of a
$-(K_X+\Delta)$-MMP with scaling of an ample divisor for pklt pairs and obtained an anticanonical minimal model in dimension two.
In a related direction, \cite{KL} proved that the
geometric generic fiber of a fibration whose closed fibers are of $\varepsilon$-lc log Calabi--Yau type is pklt, and that both the
anticanonical MMP and the $D$-MMP can be run on the geometric generic fiber for any big $\Q$-Cartier divisor $D$.

The purpose of this paper is twofold. First, we give an alternative proof of Theorem \ref{CJKL} under the assumption that the anticanonical divisor admits a birational Zariski decomposition. Our proof of Theorem \ref{CJKL} does not rely on the existence of a quasi-monomial valuation computing the log canonical threshold (cf. \cite[Question 6.15]{Leh14} and \cite[Theorem 1.1]{CJKL25}). 

\begin{theorem}[{cf. \cite[Corollary 1.4]{CJKL25}}] \label{CJKL}
Let $(X,\Delta)$ be a pklt pair such that $-(K_X+\Delta)$ admits a  birational Zariski decomposition. Then there exists a $-(K_X+\Delta)$-MMP with scaling of an ample divisor. Moreover, if $\dim X=2$, then there exists an anticanonical minimal model of $(X,\Delta)$.
\end{theorem}

The second contribution concerns the structure of the anticanonical
MMP at the level of $\mathbb{Q}$-factorial terminal models.
Given a partial $-(K_X+\Delta)$-MMP starting from a pklt pair containing no flips, it is natural to ask whether each step can be realized, after passing to
terminal models, by maps that are negative with respect to the
log canonical divisor upstairs. We prove that this is indeed the case and that the potential log discrepancy is preserved throughout the program.

\begin{theorem} \label{thm:main1}
Let $(X,\Delta)$ be a pklt pair, and let $\varphi\colon (X,\Delta)\to (X',\Delta')$ be a partial $-(K_X+\Delta)$-MMP that does not contain any flips. Then there exist $\Q$-factorial terminal models $p\colon (Y,\Delta_Y\coloneqq p^{-1}_*\Delta)\to (X,\Delta)$ and $q\colon (Y',\Delta_{Y'}\coloneqq q^{-1}_*\Delta')\to (X',\Delta')$, and a composition of $(K_Y+\Delta_Y)$-negative maps $\psi\colon (Y,\Delta_Y)\dashrightarrow (Y',\Delta_{Y'})$. Moreover, for every prime divisor $E$ over $X$,
$$ \overline{a}(E,X,\Delta)=\overline{a}(E,Y,\Delta_Y)=\overline{a}(E,X',\Delta')=\overline{a}(E,Y',\Delta_{Y'}).$$
\end{theorem}

Theorem \ref{thm:main1} may be viewed as a generalization of the factorization theorem for anticanonical maps of Fano type varieties \cite{CHP} to the pklt setting. Note that, in proving the equalities for the potential log discrepancy, we need Lemma \ref{pklt MMP0}, and we need to use the theory of quasi-monomial valuations as in \cite{CJKL25}.

As a consequence of Theorem \ref{thm:main1}, we show that the
minimal resolutions of successive steps of the MMP on a pklt surface
are connected by a sequence of redundant blow-ups
(see Definition \ref{def:redun} and Corollary \ref{cor}),
and we characterize normal projective klt surfaces with nef anticanonical
divisor whose minimal resolution has no redundant exceptional curve
(Theorem \ref{thm:redun}).

\begin{corollary}\label{cor}
Let $(S,\Delta)$ be a pklt pair, where $S$ is a normal projective surface. Let $\varphi\colon (S,\Delta)\dashrightarrow (S',\Delta')$ be a $-(K_S+\Delta)$-MMP, and let $p\colon Y\to S$ and $q\colon Y'\to S'$ be the minimal resolutions. Then there exists a commutative diagram
\[
\begin{tikzcd}
(Y,\Delta_Y) \arrow[r,"\psi"] \arrow[d,"p"] & (Y',\Delta_{Y'}) \arrow[d,"q"] \\
(S,\Delta) \arrow[r,dashed,"\phi"] & (S',\Delta')
\end{tikzcd}
\]
Moreover, $\psi\colon (Y,\Delta_Y)\to (Y',\Delta_{Y'})$ is a sequence of redundant blow-ups.
\end{corollary}

The following theorem is proved by adapting the argument of \cite[Theorem 1.2]{HP}.

\begin{theorem}\label{thm:redun}
Let $S$ be a normal projective klt surface with nef $-K_S$, and $g\colon S'\to S$ its minimal resolution. Then $S'$ has no redundant point if and only if either $S$ has at worst canonical singularities or the dual graph of the exceptional locus of $g$ is as follows.

\begin{center}
\begin{tikzpicture}[line cap=round,line join=round,x=1.0cm,y=1.0cm]
\clip(-4,2.8) rectangle (9,4.2);
\draw (-3.5,3.96)-- (-2.8,3.96);
\draw (-1.57,3.96)-- (-0.87,3.96);
\draw (-4,3.82) node[anchor=north west] {$-\underbrace{2\text{  }-2 \text{   }\text{   } -2}_{ \text{ } \text{ } \text{ } \alpha \text{ } (\alpha \geq 1)}$};
\draw (-1.32, 3.82) node[anchor=north west] {$-3$};
\draw (-2.5,4.12) node[anchor=north west] {$\cdots$};
\draw (0.2,3.96)-- (2.3,3.96);
\draw (-0.25,3.82) node[anchor=north west] {$-2$};
\draw (0.45,3.82) node[anchor=north west] {$-2$};
\draw (1.15,3.82) node[anchor=north west] {$-3$};
\draw (1.85,3.82) node[anchor=north west] {$-2$};
\draw (3.37,3.96)-- (4.77,3.96);
\draw (2.92,3.82) node[anchor=north west] {$-2$};
\draw (3.62,3.82) node[anchor=north west] {$-3$};
\draw (4.32,3.82) node[anchor=north west] {$-2$};
\draw (5.84,3.96)-- (6.54,3.96);
\draw (5.39,3.82) node[anchor=north west] {$-2$};
\draw (6.09,3.82) node[anchor=north west] {$-4$};
\draw (7.16,3.86) node[anchor=north west] {$-n$};
\draw (6.9,3.5) node[anchor=north west] {$(n \geq 3)$};
\begin{scriptsize}
\fill [color=black] (-3.5,3.96) circle (2.5pt);
\fill [color=black] (-2.8,3.96) circle (2.5pt);
\fill [color=black] (-1.57,3.96) circle (2.5pt);
\fill [color=black] (-0.87,3.96) circle (2.5pt);
\fill [color=black] (0.2,3.96) circle (2.5pt);
\fill [color=black] (0.9,3.96) circle (2.5pt);
\fill [color=black] (1.6,3.96) circle (2.5pt);
\fill [color=black] (2.3,3.96) circle (2.5pt);
\fill [color=black] (3.37,3.96) circle (2.5pt);
\fill [color=black] (4.07,3.96) circle (2.5pt);
\fill [color=black] (4.77,3.96) circle (2.5pt);
\fill [color=black] (5.84,3.96) circle (2.5pt);
\fill [color=black] (6.54,3.96) circle (2.5pt);
\fill [color=black] (7.61,3.96) circle (2.5pt);
\end{scriptsize}
\end{tikzpicture}
\end{center}

\end{theorem}

Corollary \ref{cor} and Theorem \ref{thm:redun} are inspired by the results on redundant blow-up in \cite{HP}, which are extended here beyond the setting of rational surfaces with big anticanonical divisor.

The rest of this paper is organized as follows.
In Section \ref{sect:prelim}, we recall the necessary background on
valuations, the relative Nakayama asymptotic order, and potentially
klt pairs. Section \ref{sect:main} is devoted to the proofs of the main results.
Finally, in Section \ref{sect:exa}, we present examples illustrating two different phenomena. Example \ref{ex:1} shows that an anticanonical minimal model of a potentially klt surface need not be a Mori dream space. Examples \ref{ex:2}--\ref{ex:4} show that potentially klt varieties need not be of Calabi--Yau type.

\section{Preliminaries}\label{sect:prelim}

We collect the notions used throughout the paper:

\begin{itemize}
    \item A \emph{variety} is a separated, finite type, and integral scheme over an algebraically closed field $k$ of characteristic $0$.
    \item A \emph{couple} $(X,\Delta)$ is a normal variety $X$ with an effective $\Q$-Weil divisor $\Delta$ on $X$. A couple $(X,\Delta)$ is said to be a \emph{pair} if $K_X+\Delta$ is $\Q$-Cartier.
    \item Let $(X,\Delta)$ be a pair, and let $f\colon Y\to X$ be a proper birational morphism with $Y$ normal. Let $E$ be a prime divisor on $Y$. We denote the \emph{log discrepancy} by
    $$ A_{X,\Delta}(E)\coloneqq \mathrm{mult}_E(K_Y-f^*(K_X+\Delta))+1. $$
    Note that the notion does not depend on the choice of $f$.
\end{itemize}

\subsection{Valuations}

Let us recall the notion of quasi-monomial valuations. For more details, see \cite{JM12,Xu25}.

\smallskip

Throughout this section, let $X$ be a normal variety over an algebraically closed field $k$ of characteristic $0$, and write
$K \coloneqq  k(X)$ for its function field. We work with \emph{(real) valuations}
$\nu\colon K^\times \to \mathbb{R}$ that are trivial on $k^\times$, and extend by
$\nu(0)\coloneqq +\infty$.

\smallskip

A (real) valuation on $K$ is a map $\nu\colon K^\times\to\mathbb{R}$ such that:
\begin{itemize}
  \item $\nu(a)=0$ for all $a\in k^\times$;
  \item $\nu(fg)=\nu(f)+\nu(g)$ for all $f,g\in K^\times$;
  \item $\nu(f+g)\ge \min\{\nu(f),\nu(g)\}$ for all $f,g\in K$.
\end{itemize}
The associated valuation ring is
$$
\mathcal{O}_\nu \coloneqq  \{f\in K \mid \nu(f)\ge 0\}, \qquad
\mathfrak{m}_\nu \coloneqq  \{f\in K \mid \nu(f)>0\}.
$$

We say that $\nu$ is \emph{centered on $X$} (or \emph{a valuation over $X$}) if there exists an affine open
$U=\Spec R \subseteq X$ such that $R\subseteq \mathcal{O}_\nu$.
In this case, the \emph{center} of $\nu$ on $X$ is the (not necessarily closed) point $c_X(\nu) \in X \text{ corresponding to the prime ideal } R\cap \mathfrak{m}_\nu \subseteq R$. We write $\Val_X \coloneqq  \{ \nu \text{ valuation over } X\}, \Val_{X,x} \coloneqq  \{ \nu\in \Val_X \mid c_X(\nu)=x\}$ and $\Val^*_X \coloneqq  \Val_X\setminus\{0\}$.

Let $f\colon Y\to X$ be a proper birational morphism with $Y$ normal, and  $E\subset Y$ a prime divisor.
Then $\ord_E\colon K^\times\to \mathbb{Z}\subset\mathbb{R}$ is a valuation over $X$.
Any valuation of the form $c\cdot \ord_E$ for $c>0$ is called \emph{divisorial}.

\smallskip

A \emph{log-smooth model} of $X$ is a proper birational morphism
$f\colon (X',E)\to X$ such that $X'$ is smooth and $E=\bigcup_{i=1}^r E_i$ is a reduced simple normal crossings divisor on $X'$, with the property that $f$ is an isomorphism over $X\setminus f(E)$.

Fix a log-smooth model $f\colon (X',E=\bigcup_{i=1}^r E_i)\to X$.
Assume $\bigcap_{i=1}^r E_i\neq\varnothing$, and let $C$ be a connected component of $\bigcap_{i=1}^r E_i$ with generic point $\eta$.
Since $X'$ is smooth and $E$ is snc, the local ring $\mathcal{O}_{X',\eta}$ is a regular local ring of dimension $r$; we may choose regular parameters $z_1,\dots,z_r\in \mathcal{O}_{X',\eta}$ such that $E_i=(z_i=0)$ for all $i$.

Let $\alpha=(\alpha_1,\dots,\alpha_r)\in \R_{\ge 0}^r$.
For $g\in \mathcal{O}_{X',\eta}$, write its expansion in the completed local ring $\widehat{\mathcal{O}}_{X',\eta}\simeq k(\eta)[[z_1,\dots,z_r]]$ as
$$
g=\sum_{\beta\in \mathbb{Z}_{\ge 0}^r} c_\beta z^\beta,
\text{ where } z^\beta\coloneqq z_1^{\beta_1}\cdots z_r^{\beta_r}.
$$
Define
$$
\nu_{(X',E),\eta,\alpha}(g)
\coloneqq \min\left\{\sum_{i=1}^r \alpha_i\beta_i \ \middle|\ c_\beta\neq 0 \right\}\in \mathbb{R}_{\ge 0}\cup\{+\infty\}.
$$
This gives a valuation on $\mathcal{O}_{X',\eta}$, hence it extends uniquely to a valuation on $K$.
The resulting valuation is independent of the choice of the parameters $z_i$ as above.
Valuations obtained in this way are called \emph{quasi-monomial valuations}.

We denote by $\mathrm{QM}(X',E)$ the set of quasi-monomial valuations defined from $(X',E)$,
and by $\mathrm{QM}_\eta(X',E)\subseteq \mathrm{QM}(X',E)$ those with center $c_Y(\nu)=\eta$.

\smallskip
Let $\nu\in \Val_X$ and let $\mathfrak{a}\subset \mathcal{O}_X$ be an ideal sheaf.
Set $x\coloneqq c_X(\nu)$ and define
$$
\nu(\mathfrak{a}) \coloneqq  \min\{\nu(f)\mid f\in \mathfrak{a}_x\}.
$$
Then for ideals $\mathfrak{a},\mathfrak{b}$ in $\mathcal{O}_X$, one has $\nu(\mathfrak{a}\mathfrak{b})=\nu(\mathfrak{a})+\nu(\mathfrak{b})$.

\smallskip

The usual topology on $\Val_X$ is the weakest topology such that, for every ideal $\mathfrak{a}$ in $\mathcal{O}_X$,
the function $\nu\mapsto \nu(\mathfrak{a})$ is continuous.

\smallskip

Given a log-smooth model $(X',E=\bigcup_{i=1}^r E_i)\to X$, there is a natural ``retraction'' map
$$
\rho_{(X',E)}\colon \Val_X \longrightarrow \mathrm{QM}(X',E),
\,
\rho_{(X',E)}(\nu)\coloneqq \nu_{(X',E),\,c_{X'}(\nu),\,(\nu(E_1),\dots,\nu(E_r))},
$$
which records the values of $\nu$ along the components of $E$ into quasi-monomial data on $(X',E)$.

Assume for the moment that $(X,\Delta)$ is a klt pair.
For a prime divisor $F$ over $X$, the log discrepancy $A_{X,\Delta}(F)$ is defined in the usual way.
For quasi-monomial valuations on a log-smooth model, the discrepancy is linear in the weights.

Let $f\colon (X',E=\bigcup_{i=1}^r E_i)\to X$ be a log-smooth model, let $\eta$ be the generic point of a stratum
$\bigcap_{i=1}^r E_i$, and let $\alpha\in \R_{\ge 0}^r$.
Define
$$
A_{X,\Delta}\left(\nu_{(X',E),\eta,\alpha}\right)
\coloneqq \sum_{i=1}^r \alpha_i\, A_{X,\Delta}(E_i).
$$
This is independent of the chosen log-smooth model representing the given quasi-monomial valuation (cf. \cite[Lemma 3.6]{JM12} for the compatibility of quasi-monomial representations).

\smallskip

For a general valuation $\nu\in \Val_X$, define
$$
A_{X,\Delta}(\nu)
\coloneqq \sup_{(X',E)\ \text{log-smooth over }X} 
A_{X,\Delta}\left(\rho_{(X',E)}(\nu)\right)\in [0,+\infty].
$$

\smallskip

Let $\Phi\subseteq \Z_{\ge 0}$ be a subsemigroup, and let $\mathfrak{a}_\bullet=\{\mathfrak{a}_m\}_{m\in\Phi}$
be a \emph{graded sequence of ideals}, i.e., $\mathfrak{a}_m\mathfrak{a}_n\subseteq \mathfrak{a}_{m+n}$.
For $\nu\in \Val_X$, set
$$
\nu(\mathfrak{a}_\bullet)\coloneqq \inf_{m\in \Phi}\frac{\nu(\mathfrak{a}_m)}{m}.
$$
In characteristic $0$, one defines the log canonical threshold by
$$
\lct(X,\Delta;\mathfrak{a}_\bullet)
\coloneqq \inf_{\nu\in \Val_X^\ast}\frac{A_{X,\Delta}(\nu)}{\nu(\mathfrak{a}_\bullet)}.
$$
A key fact is that the infimum is achieved by a quasi-monomial valuation:
there exists a quasi-monomial valuation $\nu\in \Val^*_X$ such that
$$
\lct(X,\Delta;\mathfrak{a}_\bullet)=\frac{A_{X,\Delta}(\nu)}{\nu(\mathfrak{a}_\bullet)}.
$$
(cf. \cite[Theorem 1.1]{Xu20})

\subsection{Relative Nakayama's asymptotic order}

Let $f\colon X\to S$ be a projective morphism from a normal variety to a variety.

\smallskip

We recall the relative Nakayama's asymptotic order, following \cite[Section 3]{LX25}.

\begin{definition}[{Relative asymptotic order \(\sigma_E\); cf.\ \cite[Section 3]{LX25}}]\label{def:sigma}
Let $f\colon X\to S$ be as above, let $A$ be an ample/$S$ $\R$-divisor on $X$, and let $D$ be a pseudoeffective/$S$ $\R$-Cartier $\R$-divisor on $X$.
Fix a proper birational morphism $g\colon X'\to X$, and let $E$ be a prime divisor on $X'$.

\begin{itemize}
  \item If $D$ is a big/$S$ $\R$-divisor on $X$, define
  $$
    \sigma_E(X/S,D)\coloneqq \inf\Bigl\{\mult_E(D') \ \Big|\ 0\le D' \sim_{\R,S} g^*D \Bigr\}.
  $$
  \item For pseudoeffective/$S$ $\R$-divisor $D$, define
  $$
    \sigma_E(X/S,D)\coloneqq \lim_{\varepsilon\to 0^+} \sigma_E(X/S, D+\varepsilon A),
  $$
  allowing $+\infty$ as a limit.
\end{itemize}
It is known that $\sigma_E(X/S,D)$ is well-defined and independent of the choice of the ample/$S$ divisor $A$.
\end{definition}

We denote by $\mathfrak{b}(X/S,|L|)$ the relative base ideal of the linear system $|L|$ over $S$.

\begin{definition}
    Let $f\colon X\to S$ be as above, $D$ a pseudoeffective/S $\mathbb Q$-Cartier $\Q$-divisor on $X$, and $\nu\in \mathrm{Val}_X$.
If $D$ is big over $S$ and $nD$ is Cartier for some $n>0$, define
\[
\sigma_\nu(X/S,D)\coloneqq \frac{1}{n}\inf_{m\ge 1}\frac{\nu(\mathfrak b(X/S,|mnD|))}{m}.
\]
If $D$ is only pseudoeffective over $S$, choose an ample/$S$ divisor $A$ on $X$
and define
\[
\sigma_\nu(X/S,D)\coloneqq \lim_{\varepsilon\to 0^+}\sigma_\nu(X/S,D+\varepsilon A).
\]
Note that this limit exists and is independent of the choice of $A$.
\end{definition}

\begin{remark}
If $\nu=c\cdot \ord_E$ for a prime divisor $E$ over $X$ and $c>0$, then
\[
\sigma_\nu(X/S,D)=c\,\sigma_{\ord_E}(X/S,D).
\]
\end{remark}

\begin{definition}
Let $f\colon X\to S$ be as above, let $(X,\Delta)$ be a klt pair, and let
$D$ be a pseudoeffective/$S$ $\Q$-Cartier $\Q$-divisor on $X$.
We define
\[
\lct_\sigma(X/S,\Delta,D)\coloneqq 
\inf_{\nu\in \mathrm{Val}_X^*}\frac{A_{X,\Delta}(\nu)}{\sigma_\nu(X/S,D)},
\]
with the convention that the quotient is $+\infty$ when $\sigma_\nu(X/S,D)=0$.
\end{definition}

If the base $S$ is just a point, then we simply write $\sigma_E(D)\coloneqq \sigma_E(X/S,D)$, $\sigma_\nu(D)\coloneqq \sigma_\nu(X/S,D)$ and $\lct_\sigma(X,\Delta,D)\coloneqq \lct_\sigma(X/S,\Delta,D)$.

The \textit{negative part} $N_{\sigma}(D)$ of $D$ is defined as
\begin{align*}
  N_{\sigma}(D)\coloneqq \sum_E \sigma_E(D)E,
\end{align*}
and the \textit{positive part} $P_{\sigma}(D)$ of $D$ is defined as $P_{\sigma}(D)\coloneqq D-N_{\sigma}(D)$. We call $D=P_{\sigma}(D)+N_{\sigma}(D)$ the \textit{divisorial Zariski decomposition} of $D$. We call it the \textit{Zariski decomposition} if $P_{\sigma}(D)$ is nef. Moreover, if there exists a projective birational morphism $f\colon Y\rightarrow X$ such that $P_{\sigma}(f^{\ast}D)$ is nef, then we say that $D$ admits a \textit{birational Zariski decomposition}.

\begin{definition}
Let $f\colon X\to Z$ be a projective morphism. We say that $X$ is of
\emph{Fano type over $Z$} if there exists an effective $\Q$-divisor $\Delta$
on $X$ such that $(X,\Delta)$ is klt and $-(K_X+\Delta)$ is ample over $Z$.
\end{definition}

\begin{definition}
A normal projective variety $X$ is said to be of \emph{Calabi--Yau type}
if there exists an effective $\Q$-divisor $\Delta$ such that $(X,\Delta)$ is log canonical and $K_X+\Delta\sim_{\Q} 0$.
\end{definition}

\begin{lemma}
Let $(X,\Delta)$ be a klt pair and $f\colon X\to Z$ be a morphism. Suppose that $-(K_X+\Delta)$ is big over $Z$. If $\lct_{\sigma}(X/Z,\Delta,-(K_X+\Delta))>1$, then $X$ is of Fano type over $Z$.
\end{lemma}

\begin{proof}
Let $n$ be a positive integer such that $n(K_X+\Delta)$ is Cartier and $f_*\mathcal{O}_X(-n(K_X+\Delta))\ne 0$. Define $\mathfrak{b}_{m}\coloneqq \mathfrak{b}(X/Z,|-mn(K_X+\Delta)|)$.
Then $\mathfrak{b}_{\bullet}\coloneqq \{\mathfrak{b}_m\}_{m\ge 1}$ is a graded sequence of ideals in $\mathcal{O}_X$, and $\mathcal{J}(X,\Delta,\mathfrak{b}^{\frac{1}{n}}_{\bullet})=\mathcal{O}_X$ if and only if $\lct_{\sigma}(X/Z,\Delta,-(K_X+\Delta))>1$ by \cite[Lemma 1.60]{Xu25}. On the other hand, $\mathcal{J}(X,\Delta,\mathfrak{b}^{\frac{1}{n}}_{\bullet})=\mathcal{O}_X$ if and only if there exists an effective $\Q$-divisor $\Delta'\sim_{\Q,Z} -(K_X+\Delta)$ such that $(X,\Delta+\Delta')$ is klt. Hence, $(X,\Delta)$ is of Fano type over $Z$.
\end{proof}

\subsection{Potentially klt pairs}

We now recall the notion of potentially klt pairs.

\begin{definition}
    Let $(X,\Delta)$ be a pair such that $D\coloneqq -(K_X+\Delta)$ is pseudoeffective.
For a prime divisor $E$ over $X$, define the \emph{potential log discrepancy} by
\[ 
\bar a(E;X,\Delta)\coloneqq A_{X,\Delta}(E)-\sigma_E(D),
\]
where $A_{X,\Delta}(E)$ is the \emph{log discrepancy} of $E$ with respect to $(X,\Delta)$.
We say that $(X,\Delta)$ is \emph{potentially klt} (or \emph{pklt}) if $\inf_{E}\bar a(E;X,\Delta)>0$, where the $\inf$ is taken over all prime divisors $E$ over $X$.
\end{definition}

Let us define the notion of a $\Q$-factorial terminal model.

\begin{definition}
Let $(X,\Delta)$ be a klt pair, and let $f\colon (Y,\Delta_Y)\to (X,\Delta)$ be a proper birational morphism with $\Delta_Y=f^{-1}_*\Delta$. We say $f$ is a \emph{$\Q$-factorial terminal model} if
\begin{itemize}
    \item[(a)] $(Y,\Delta_Y)$ is a \(\Q\)-factorial terminal pair, and
    \item[(b)] $K_Y+\Delta_Y$ is nef over $X$.
\end{itemize}
\end{definition}

\begin{definition}
Let $X$ be a normal projective variety and $D$ a $\Q$-Cartier divisor on $X$. A birational contraction $\varphi\colon X\dashrightarrow X'$
is called \emph{$D$-nonpositive} if $\varphi_{\ast}D$ is a $\Q$-Cartier divisor on a normal projective variety $X'$, and if there exists a common resolution
\[
\xymatrix{
& W \ar[dl]_p \ar[dr]^q & \\
X \ar@{-->}[rr]^{\varphi} && X'
}
\]
such that $p^*(D)=q^*(\varphi_{\ast}D)+E$, where $E\geq 0$ is a $q$-exceptional divisor. If moreover $\Supp(E)$ contains all the strict transforms of the $\varphi$-exceptional divisors, then we say that $\varphi$ is \emph{$D$-negative}.
\end{definition}

\begin{definition}
A birational contraction $\varphi\colon (X,\Delta)\dashrightarrow (Y,\Delta_Y)$ is called an \emph{$-(K_X+\Delta)$-minimal model} of $(X,\Delta)$ if $\varphi$ is $-(K_X+\Delta)$-negative and $-(K_Y+\Delta_Y)$ is nef. By a \emph{partial $-(K_X+\Delta)$-MMP}, we mean a finite sequence of steps in a $-(K_X+\Delta)$-MMP.
\end{definition}

\begin{lemma} \label{pklt MMP}
Let $(X,\Delta)$ be a pklt pair and $\varphi\colon X\to Z$ a birational contraction. Then $(X,\Delta)$ is of Fano type over $Z$. In particular, we can run a $-(K_X+\Delta)$-MMP with scaling of an ample divisor over $Z$, and it terminates.
\end{lemma}

\begin{proof}
Since $\varphi$ is birational, $-(K_X+\Delta)$ is big over $Z$. For any prime divisor $E$ over $X$, we have $\sigma_E(-(K_X+\Delta))\ge \sigma_E(X/Z,-(K_X+\Delta))$. Since $(X,\Delta)$ is pklt, there exists $\varepsilon>0$, independent of the choice of $E$, such that $A_{X,\Delta}(E)-\sigma_E(X/Z,-(K_X+\Delta))>\varepsilon$.

\smallskip

We use Diophantine approximation to prove $\lct_{\sigma}(X/Z,\Delta,-(K_X+\Delta))>1$. Let $\nu_0\in \Val_X$ be a quasi-monomial valuation computing $\lct_{\sigma}(X/Z,\Delta,-(K_X+\Delta))$ (See \cite[Theorem 1.1]{CJKL25}), $\eta$ the center of $\nu_0$, and let $(Y,E)\to X$ be a log-smooth model of $X$ such that $\nu_0\in \mathrm{QM}_{\eta}(Y,E)$.

\smallskip

Define a function $\phi\colon \mathrm{QM}_{\eta}(Y,E)\to \R$ by $\phi(\nu)\coloneqq A_{X,\Delta}(\nu)-\sigma_{\nu}(X/Z,-(K_X+\Delta))$ for $\nu\in \mathrm{QM}_{\eta}(Y,E)$. By the concavity of the function
\[
\nu\mapsto \sigma_\nu(X/Z,-(K_X+\Delta))
\]
on $\mathrm{QM}_\eta(Y,E)$, the function $\phi$ is convex on $\mathrm{QM}_{\eta}(Y,E)$. Since every convex function is locally Lipschitz on an open and convex subset of a Euclidean space, there exist positive real numbers $C,\delta$ such that $|\phi(\nu_0)-\phi(\nu)|<C\|\nu_0-\nu\|$ for all valuations $\nu\in \mathrm{QM}_{\eta}(Y,E)$ with $\|\nu_0-\nu\|\le \delta$. By \cite[Lemma 2.7]{LX18}, for each $t>0$, there exist a divisorial valuation $\nu_t\in \mathrm{QM}_{\eta}(Y,E)$, a prime divisor $F_t$ over $X$, and a positive rational number $q_t$ such that
\begin{itemize}
    \item $q_t\cdot \nu_t=\mathrm{ord}_{F_t}$, and
    \item $\|\nu_0-\nu_t\|<\frac{t}{q_t}$.
\end{itemize}

For a sufficiently small $t>0$, we have $\phi(\ord_{F_t})>\varepsilon$. Since $\ord_{F_t}=q_t\nu_t$, by homogeneity we obtain $\phi(\nu_t)=\frac{1}{q_t}\phi(\ord_{F_t})>\frac{\varepsilon}{q_t}$. Hence, we obtain the following inequalities
\[
\phi(\nu_0)\ge \phi(\nu_t)-|\phi(\nu_0)-\phi(\nu_t)|
> \frac{\varepsilon}{q_t}-C\|\nu_0-\nu_t\|
> \frac{\varepsilon-Ct}{q_t}>0
\]
for sufficiently small $t>0$.
Therefore, we have $A_{X,\Delta}(\nu_0)-\sigma_{\nu_0}(X/Z,-(K_X+\Delta))>0$. It follows that 
$$ \lct_{\sigma}(X/Z,\Delta,-(K_X+\Delta))=\frac{A_{X,\Delta}(\nu_0)}{\sigma_{\nu_0}(X/Z,-(K_X+\Delta))}>1.$$
Applying \cite[Lemma 1.60]{Xu25}, we conclude that $X$ is of Fano type over $Z$. The second assertion follows from \cite[Theorem 5.6]{Oht22}.
\end{proof}

The next lemma explains how the pklt condition behaves under a $\Q$-factorial terminal model.

\begin{lemma} \label{pklt MMP0}
Let $(X,\Delta)$ be a pklt pair and let $f\colon (Y,\Delta_Y)\to (X,\Delta)$ be a $\Q$-factorial terminal model. Then $(Y,\Delta_Y)$ is also a pklt pair, and the potential log discrepancy is preserved:
$$ \overline{a}(E,X,\Delta)=\overline{a}(E,Y,\Delta_Y),$$
\end{lemma}

\begin{proof}
Let $D\coloneqq f^*(K_X+\Delta)-(K_Y+\Delta_Y)$. Since $f_*D=0$ and $D$ is anti $f$-nef, the negativity lemma (cf. \cite[Lemma 3.39]{KM98}) gives $D\ge 0$. We can write $f^*(K_X+\Delta)$ as
$$ f^*(K_X+\Delta)=K_Y+(\Delta_Y+D),$$
and hence $(Y,\Delta_Y+D)$ is pklt since the property ``potentially klt'' is invariant under any crepant morphism. Consequently, the pair \((Y,\Delta_Y)\) is pklt by \cite[Lemma 3.5]{CP16}.

\smallskip

By Lemma \ref{pklt MMP}, we can run a $-(K_Y+\Delta_Y)$-MMP $\varphi\colon Y\dashrightarrow Y'$ over $X$. Let $g\colon Y'\to X$ be the structure morphism, and let $\Delta_{Y'}\coloneqq \varphi_*\Delta_Y$ and $D'\coloneqq \varphi_*D$. Note that $D'$ is nef over $X$, and $g_*D'=0$. Thus, by the negativity lemma (cf. \cite[Lemma 3.39]{KM98}) again, $D'\le 0$. Hence, $D'=0$, and $g^*(K_X+\Delta)=K_{Y'}+\Delta_{Y'}$. Since the potential log discrepancy is preserved under any crepant morphism and under an anticanonical MMP,
$$ \overline{a}(E,X,\Delta)=\overline{a}(E,Y',\Delta_{Y'})=\overline{a}(E,Y,\Delta_Y),$$
and we proved the lemma.
\end{proof}

\subsection{Redundant blow-up}

We now introduce the notion of a \emph{redundant blow-up}. This notion was first defined in \cite[Definition 4.1]{Sak84} in the study of anticanonical models of rational surfaces with $\kappa(-K_S)=2$. Here, we extend it to a broader setting.

\begin{definition}\label{def:redun}
Let $(S,\Delta)$ be a smooth projective surface pair such that
$-(K_S+\Delta)$ is pseudoeffective, and let $-(K_S+\Delta)=P+N$ be the Zariski decomposition. We say that $p\in S$ is a \emph{redundant point} of $(S,\Delta)$ if $\mult_p(N+\Delta)\ge 1$. A blow-up at a redundant point is called a \emph{redundant blow-up}, and its exceptional divisor is called a \emph{redundant exceptional curve}. A smooth projective surface is said to \emph{have a redundant exceptional curve} if it contains the exceptional divisor of a redundant blow-up.
\end{definition}

The following lemma provides a Zariski decomposition based criterion for a point to be redundant. 

\begin{proposition}\label{prop:redun-equiv}
Let $f\colon \widetilde S\to S$ be a blow-up of a smooth point $p$ on a smooth projective surface pair $(S,\Delta)$, and let $\widetilde\Delta$ be the strict transform of $\Delta$. Write $-(K_S+\Delta)=P+N$ for the Zariski decomposition, and let $E$ be the $f$-exceptional curve.
Then the following are equivalent:
\begin{enumerate}
\item[\emph{(1)}] $p$ is a redundant point of $(S,\Delta)$,
\item[\emph{(2)}] $-(K_{\widetilde S}+\widetilde\Delta)
=f^*P+\bigl(f^*N+(\mult_p\Delta-1)E\bigr)$
is the Zariski decomposition of $-(K_{\widetilde S}+\widetilde\Delta)$,
\item[\emph{(3)}] $f^*N+(\mult_p\Delta-1)E$ is effective.
\end{enumerate}
\end{proposition}

\begin{proof}
Since $f\colon \widetilde{S}\to S$ is a blow-up of a smooth point,
$$ -(K_{\widetilde{S}}+\widetilde\Delta)=-f^*(K_S+\Delta)+(\mathrm{mult}_p\Delta-1)E=f^*P+f^*N+(\mathrm{mult}_p\Delta-1)E.$$
Moreover, $f^*P$ is nef, and
$$ f^*N+(\mathrm{mult}_p\Delta-1)E=f^{-1}_*N+(\mathrm{mult}_p N+\mathrm{mult}_p\Delta-1)E,$$
and it is effective if and only if $f$ is a redundant blow-up.

\smallskip

It remains to verify the orthogonality condition and negative definiteness. Orthogonality follows from $P\cdot N=0$ and $f^*P\cdot E=0$. 

Let us prove that the support of $f^*N+(\mathrm{mult}_p\Delta-1)E$ is negative definite. To see this, write
$$ N=\sum_i a_iC_i. $$
Then, $f^*C_i=C'_i+r_iE$ for some $r_i$, where $C'_i$ is the strict transform. The intersection number is $C'_i\cdot C'_j=C_i\cdot C_j-r_ir_j$, $C'_i\cdot E=r_i$, and $E^2=-1$. Then,
$$\begin{aligned}
\left(\sum_i x_iC'_i+tE\right)^2&=\sum_{i,j}x_ix_j(C_i\cdot C_j-r_ir_j)+2t\sum_i x_ir_i-t^2 
\\ &=\sum_{i,j}x_ix_j(C_i\cdot C_j)-\left(\sum_i x_ir_i-t\right)^2.
\end{aligned}$$
Since the matrix $(C_i\cdot C_j)$ is negative definite, the first term is $<0$. This completes the proof.
\end{proof}

\section{Main results and Proofs}\label{sect:main}

In this section, we give proofs of our main theorems.

\begin{proof}[{Proof of Theorem \ref{CJKL}}]
It suffices to show 
\begin{equation} \label{lct}
    \mathrm{lct}_{\sigma}(X,\Delta,-(K_X+\Delta))>1.
\end{equation} 
Once (\ref{lct}) is established, the rest of the argument follows from \cite[Proof of Corollary 1.4]{CJKL25}.

\smallskip

Let $f\colon Y\to X$ be a log resolution of $(X,\Delta)$ such that $f^*(-(K_X+\Delta))$ admits a Zariski decomposition $f^*(-(K_X+\Delta))=P+N$. Then by \cite[Proof of Corollary 1.2]{CJL25}, for sufficiently small $\varepsilon>0$, we have $A_{X,\Delta}(E)-\sigma_E((1+\varepsilon)(-(K_X+\Delta)))\ge 0$.
Therefore, we obtain
$$
A_{X,\Delta}(E)-\sigma_E(-(K_X+\Delta))\ge \varepsilon \sigma_E(-(K_X+\Delta)) 
$$
for every prime divisor $E$ on $Y$. Let $g\colon Z\to Y$ be a composition of smooth blow-ups. Let us prove
\begin{equation} \label{asdf}
A_{X,\Delta}(E)-\sigma_E(-(K_X+\Delta))\ge \varepsilon \sigma_E(-(K_X+\Delta)) 
\end{equation}
for every prime divisor $E$ on $Z$. Let us assume that $g$ is a blow-up along a smooth subvariety in $Y$. We first compute $\sigma_E(-(K_X+\Delta))$.
$$ 
\begin{aligned}
\sigma_E(-(K_X+\Delta))&=\mathrm{mult}_EN_{\sigma}((f\circ g)^*(-(K_X+\Delta))) &
\\ &=\mathrm{mult}_Eg^*N_{\sigma}(f^*(-(K_X+\Delta))& (1)
\\ &=\sum_{g(E)\subseteq E_i}\mathrm{mult}_{E_i}N_{\sigma}(f^*(-(K_X+\Delta))&
\\ &=\sum_{g(E)\subseteq E_i}\sigma_{E_i}(-(K_X+\Delta)),&
\end{aligned}$$
where (1) is derived from \cite[Lemma 3.2.5]{Nak04}. 

Next we estimate \(A_{X,\Delta}(E)\). If we let $K_Y+\Delta_Y=f^*(K_X+\Delta)$,
then
$$
\begin{aligned}
A_{X,\Delta}(E)&=\mathrm{mult}_EK_Z-\mathrm{mult}_Eg^*(K_Y+\Delta_Y)+1
\\ &=\mathrm{mult}_E K_Z-\mathrm{mult}_Eg^*K_Y-\sum_{g(E)\subseteq E_i}\mathrm{mult}_E\Delta_Y+1
\\ &=\mathrm{mult}_Eg^*K_Y+A_Y(E)-\mathrm{mult}_Eg^*K_Y-\sum_{g(E)\subseteq E_i}\mathrm{mult}_E\Delta_Y
\\ &=A_Y(E)-\sum_{g(E)\subseteq E_i}\mathrm{mult}_{E_i}\Delta_Y
\\ &=A_Y(E)-\sum_{g(E)\subseteq E_i}1+\sum_{g(E)\subseteq E_i}A_{X,\Delta}(E_i)
\\ &\ge \sum_{g(E)\subseteq E_i}A_{X,\Delta}(E_i).
\end{aligned}
$$
Here, we used that the center of $E$ on $Y$ is contained in at most $A_Y(E)$ components of the simple normal crossings divisor supporting $\Delta_Y$.

Hence, we have the following inequalities
$$ 
\begin{aligned}
A_{X,\Delta}(E)-\sigma_E(-(K_X+\Delta))&\ge \sum_{g(E)\subseteq E_i}(A_{X,\Delta}(E_i)-\sigma_{E_i}(-(K_X+\Delta)))
\\ &\ge \varepsilon\sum_{g(E)\subseteq E_i}\sigma_{E_i}(-(K_X+\Delta))
\\ &=\varepsilon\sigma_E(-(K_X+\Delta)),
\end{aligned}$$
which establish (\ref{asdf}). The general case follows in the same way.

\smallskip

Thus, by \cite[Lemma 1.60]{Xu25}, the infimum in the definition of $\lct_\sigma(X,\Delta,-(K_X+\Delta))$ may be taken over divisorial valuations. Hence, by the above computations,
\[
\lct_\sigma(X,\Delta,-(K_X+\Delta))
=\inf_E\frac{A_{X,\Delta}(E)}{\sigma_E(-(K_X+\Delta))}
\ge 1+\varepsilon>1.
\]

\smallskip

If $\dim X=2$, then there are no flips in the MMP, and therefore every sequence of the MMP must terminate.
\end{proof}

\begin{proof}[{Proof of Theorem \ref{thm:main1}}]
Let
$$ (X_0,\Delta_0)\coloneqq (X,\Delta)\overset{\varphi_1}{\to} (X_1,\Delta_1)\overset{\varphi_2}{\to} \cdots \overset{\varphi_n}{\to} (X_n,\Delta_n)\coloneqq (X',\Delta')$$
be a finite sequence of partial $-(K_X+\Delta)$-MMP that does not contain any flips.

\smallskip

Let $(Y_0,\Delta'_{Y_0}\coloneqq p^{-1}_*\Delta_0)\to (X_0,\Delta_0)$ be a $\Q$-factorial terminal model (it exists by \cite[Lemma 1.33]{Kol13}). As in \cite[Proof of Corollary 4.2.5]{CHP}, we can run a $(K_{Y_0}+p^{-1}_*\Delta_0)$-MMP over $(X_1,\Delta_1)$. Repeating the argument gives $(Y',\Delta_{Y'})\to (X',\Delta')$, and
$$ \overline{a}(E,X,\Delta)=\overline{a}(E,Y,\Delta_Y)=\overline{a}(E,X',\Delta')=\overline{a}(E,Y',\Delta_{Y'})$$
for every prime divisor $E$ over $X$ (cf. Lemma \ref{pklt MMP0}).
\end{proof}

\begin{proof}[Proof of Corollary \ref{cor}]
Since there are no flips in dimension two, Theorem \ref{thm:main1} applies. The induced birational map $Y\dashrightarrow Y'$ factors as a sequence of blow-ups at smooth points. Hence, it is enough to show that each such blow-up is redundant.

Let $\psi_i\colon Y_i\to Y_{i+1}$ be one morphism in this factorization, with exceptional curve \(E_i\), and let $u_i\coloneqq \psi_i(E_i)\in Y_{i+1}$.
Let $-(K_{Y_{i+1}}+\Delta_{i+1})=P_i+N_i$ be the Zariski decomposition.

By the preservation of potential log discrepancies along upstairs factorization in Theorem \ref{thm:main1}, we obtain that
$$ 
\begin{aligned}
\overline{a}(E_i,Y_i,\Delta_i)&=1-\sigma_{E_i}(-(K_{Y_{i}}+\Delta_{Y_{i}}))
\\ &=(2-\mathrm{mult}_{u_i}\Delta_{i+1})-\sigma_{E_i}(-(K_{Y_{i+1}}+\Delta_{Y_{i+1}}))
\\ &=\overline{a}(E_i,Y_{i+1},\Delta_{Y_{i+1}}),
\end{aligned}$$
and hence $\mathrm{mult}_{u_i}N_i+\mathrm{mult}_{u_i}\Delta_{Y_{i+1}}\ge 1$. Thus, by the definition of redundant blow-up, we obtain that every $\psi_i$ is redundant, and hence, $\psi\colon (Y,\Delta_Y)\to (Y',\Delta_{Y'})$ is a sequence of redundant blow-ups.
\end{proof}

\begin{proof}[Proof of Theorem \ref{thm:redun}]
Let $g\colon S'\to S$ be the minimal resolution. Since every klt surface has rational singularities, the support of the exceptional locus of $g$ has simple normal crossings. We now follow the proof of \cite[Theorem 1.2]{HP}. The argument there uses only the simple normal crossings property of the exceptional locus and does not use the rationality assumption in the statement. Hence, the same proof applies verbatim in our setting.
\end{proof}

\section{Examples}\label{sect:exa}

\begin{example} [{cf. \cite[Theorem 4.1]{Z}}] \label{ex:1}
	Let $\zeta$ be a primitive cube root of unity, and let
\[
[1,\zeta^i,\zeta^j]\quad (i,j\in\{0,1,2\}),\qquad [1,0,0],\ [0,1,0],\ [0,0,1]
\]
be the $12$ points in $\mathbb P^2$.
Let $\pi\colon X\to \mathbb P^2$ be the blow-up of these points, $H\coloneqq \pi^*\mathcal O_{\mathbb P^2}(1)$, and $E_1,\dots,E_{12}$ the exceptional divisors.

There are $9$ lines $L_1,\dots,L_9$ in $\mathbb P^2$, each of which passes through exactly $4$ of the above $12$ points. Let $\widetilde L_i$ be the strict transform of $L_i$ on $X$. Then the curves $\widetilde L_1,\dots,\widetilde L_9$ are pairwise disjoint $(-3)$-curves, and $\sum_{i=1}^9 \widetilde L_i \sim 9H-3\sum_{j=1}^{12}E_j = -3K_X$. Hence, we have $-K_X \sim_{\mathbb Q} \frac13\sum_{i=1}^9 \widetilde L_i$. By contracting all the $\widetilde{L_i}$, we obtain a klt $-K_X$-minimal model $\varphi\colon X\rightarrow Y$, where $Y$ is a klt Calabi--Yau surface with nine singular points of type $\frac{1}{3}(1,1)$ and $\rho(Y)=4$. Furthermore, since $-3K_Y\sim 0$, $-K_Y$ is semiample.

Moreover, the nef cone $\mathrm{Nef}(Y)$ is circular, as explained in \cite[Example, p. 245]{Totaro2012}. Since $X$ is rational and $\varphi$ is birational, the surface $Y$ is also rational, which implies that $h^1(Y,\mathcal{O}_Y)=0$. Hence, by \cite[Corollary 5.1]{Totaro2012}, the Cox ring $\mathrm{Cox}(Y)$ is not finitely generated, and thus $Y$ is not a Mori dream space. Since $\varphi\colon X\rightarrow Y$ is a surjective morphism, $X$ is also not a Mori dream space by \cite[Theorem 1.1]{Okawa2016}.

In particular, this example shows that, even if a variety $X$ is not a Mori dream space, $X$ may admit an anticanonical minimal model.
\end{example}

\begin{example}\label{ex:2}
    Let $S$ be a smooth projective rational surface as in \cite[Example~4.8]{CP16}, coming from Nikulin's $(*** )$-surfaces \cite[p.~84]{Nik}.
Then we have $-K_S$ nef and $\kappa(-K_S)=0$, and for every $m>0$, $|-mK_S|=\{mD\}$ for a unique effective divisor $D=\sum_{i=1}^9 a_i C_i \in |-K_S|$. Moreover, $D$ has a component with coefficient $>1$; hence, $S$ is not of Calabi--Yau type.

Choose a member of this family whose affine Dynkin support is of type $\widetilde E_8$.
The affine marks are $(1,2,3,4,5,6,4,3,2)$. Choose adjacent components $C_r,C_s$ with $a_r=6, a_s=5$. Set $c\coloneqq a_r+a_s=11$.

Let $f \colon X\coloneqq \mathrm{Bl}_p S \to S$, where $p\coloneqq C_r\cap C_s$ and denote by $E$ the exceptional curve and by $C_i'$ the strict transform of $C_i$. Since $f^*D=\sum_{i=1}^9 a_i C_i' + cE$, we obtain 
\[
-K_X=f^*(-K_S)-E
     \sim D_X\coloneqq \sum_{i=1}^9 a_i C_i' + (c-1)E
     = \sum_{i=1}^9 a_i C_i' + 10E.
\]
Since $D\cdot C_i=0$ for every $i$, one has $(-K_X)\cdot C_r' = (-K_S)\cdot C_r - 1 = -1$. Similarly, one has $(-K_X)\cdot C_s' = (-K_S)\cdot C_s - 1 = -1$. These show that $-K_X$ is not nef.

We next show that $|-mK_X|=\{mD_X\}$ for $m>0$. Indeed, we have the following inclusion $H^0(X,-mK_X)=H^0\!\bigl(S,-mK_S\otimes I_p^m\bigr)\subseteq H^0(S,-mK_S)$, and the right-hand side is one-dimensional, generated by the section defining $mD$.
Since $mD$ vanishes at $p$ to order $mc\ge m$, it belongs to $H^0(S,-mK_S\otimes I_p^m)$, and hence the inclusion is an equality.

Therefore, any effective $\Delta\sim_{\mathbb Q}-K_X$ must satisfy $\Delta=D_X$: after clearing denominators, one has $m\Delta\in |-mK_X|=\{mD_X\}$. However, the coefficient of $E$ in $D_X$ is $10>1$, hence, $(X,D_X)$ is not log canonical. Thus, $X$ is not of Calabi--Yau type.

Moreover, we have the Zariski decomposition $-K_X=P+N$, where $P=\frac{10}{11}f^*D$ and $N=\frac{1}{11}\sum_{i=1}^9 a_i C_i'$. The intersection matrix $(C_i'\cdot C_j')_{i,j}$ is negative definite. Hence, by Artin's contraction criterion, contracting all the irreducible curves in $\Supp(N)$ gives a  birational morphism $g\colon X\to Y$. Since \(P\cdot C_i'=0\) for every component of \(\Supp(N)\), the nef divisor \(P\) descends to a nef divisor on \(Y\), which we denote by \(-K_Y\). Thus, $-K_X=g^*(-K_Y)+N$. Since all coefficients of \(N\) are \(<1\), the surface \(Y\) is klt. Therefore, \(Y\) is the \(-K_X\)-minimal model. Finally,
\cite[Corollary 3.12]{CP16} implies that \((X,0)\) is potentially klt.
\end{example}

\begin{example}\label{ex:3}

We construct an explicit example of a smooth projective rational surface that is potentially klt but not of Calabi--Yau type, whose $-K_X$-MMP is nontrivial and terminates at a klt surface.

Let $S$ be a non-fibered generalized Halphen surface in the irreducible additive
case (Add$_1$ in the terminology of \cite{Favre}), so that $(S,D)$ is a
generalized rational Okamoto--Painlev\'e pair with $D\in|-K_S|$ the unique
anticanonical divisor, where $D$ is an irreducible cuspidal cubic
(see \cite[\S 6]{SaitoTakebeTerajima} and \cite{Favre}).
Since $S$ is obtained by blowing up $\P^2$ at $9$ points, one has $K_S^2=0, D^2=(-K_S)^2=0$. Also, note that \(D\) is nef.

We first show that $h^0(S,\mathcal{O}_S(mD))=1$ for all $m>0$.
Setting $L\coloneqq \mathcal{O}_D(D)\in\Pic^0(D)$, the uniqueness of $D$ in $|-K_S|$ implies $L\not\simeq\mathcal{O}_D$. Since $D$ is a cuspidal cubic, $\Pic^0(D)\simeq\mathbb{G}_a$, which in characteristic zero has no nontrivial torsion, and hence, $L^{\otimes m}\not\simeq\mathcal{O}_D$ for any $m>0$. As $L^{\otimes m}$ has degree $0$ on the integral curve $D$, this forces $H^0(D,L^{\otimes m})=0$ for all $m>0$.
Applying the exact sequence
$$0\to\mathcal{O}_S((m-1)D)\to\mathcal{O}_S(mD)\to\mathcal{O}_D(mD)\to 0$$
and arguing by induction on $m$, we obtain $h^0(S,\mathcal{O}_S(mD))=h^0(S,\mathcal{O}_S)=1$ for any $m>0$, or equivalently, $|-mK_S|=\{mD\}$ for all $m>0$.

Now let $f\colon X\coloneqq \mathrm{Bl}_p S\to S$ be the blow-up at the cusp $p\in D$, with exceptional curve $E$ and strict transform $C$ of $D$.
Since $p$ is a cusp of multiplicity $2$, we have $f^*D=C+2E$, and hence $-K_X=f^*(-K_S)-E\sim C+E\eqqcolon D_X$. The unique section defining $mD$ vanishes at the cusp to order $2m\ge m$, and hence $|-mK_X|=\{mD_X\}$ for any $m>0$. The relevant intersection numbers are $C^2=D^2-4=-4, E^2=-1$ and $C\cdot E=2$.
Setting $P\coloneqq \tfrac{1}{2}f^*D=\tfrac{1}{2}C+E$ and $N\coloneqq \tfrac{1}{2}C$, we obtain that $P+N=C+E=-K_X$. Since $D$ is nef on $S$, the divisor $P$ is nef on $X$, and $P\cdot C=0$,
which implies that $-K_X=P+N$ is the Zariski decomposition.
In particular, we have $(-K_X)\cdot C=(C+E)\cdot C=-4+2=-2<0$, which implies that $-K_X$ is not nef.

We next show that $X$ is not of Calabi--Yau type. Any effective $\mathbb{Q}$-divisor $\Delta\sim_{\mathbb{Q}}-K_X$ must equal
$D_X=C+E$ by the uniqueness of $|-mK_X|$. Near the intersection point $q\coloneqq C\cap E$, choose analytic coordinates $(u,v)$
with $E=\{u=0\}$ and $C=\{u-v^2=0\}$, so that $D_X$ is locally defined by
$u(u-v^2)=0$, a tacnode singularity. For the divisorial valuation $F$ associated with the weighted blow-up of weights $(2,1)$, one computes $A_X(F)=3$ and $\ord_F\!\bigl(u(u-v^2)\bigr)=4$. Indeed, \(\ord_F(u)=2\), \(\ord_F(v)=1\), and hence $\ord_F(u-v^2)=2$. Therefore, $\ord_F\!\bigl(u(u-v^2)\bigr)=4$, which implies that $A_{X,D_X}(F)=A_X(F)-\ord_F(D_X)=3-4=-1<0$. Hence, $(X,D_X)$ is not log canonical, and $X$ is not of Calabi--Yau type.

Finally, since $C$ is a smooth rational curve with $C^2=-4$, it can be
contracted by a morphism $g\colon X\to Y$.
The relation $-K_X=P+\tfrac{1}{2}C$ then gives $g^*(-K_Y)=P$, or equivalently $K_X=g^*K_Y-\tfrac{1}{2}C$, which shows that $Y$ is klt and $-K_Y$ is nef. Thus, $g$ is the unique nontrivial step of the $-K_X$-MMP, and $Y$ is the $-K_X$-minimal model. By \cite[Corollary 3.12]{CP16}, the pair $(X,0)$ is potentially klt.
\end{example}

\begin{example}\label{ex:4}

We construct a threefold example using the previous example.

Let $X$ and $g\colon X\to Y$ be as above, and let $B$ be an elliptic curve.
Set $W\coloneqq X\times B, Z\coloneqq Y\times B$ and $G\coloneqq g\times\mathrm{id}_B\colon W\to Z$, and let $p_1\colon W\to X$ denote the first projection. Since $K_B\sim 0$, one has $K_W=p_1^*K_X$ and $K_Z=p_1^*K_Y$, and hence, we obtain that $-K_W=p_1^*(-K_X)$ and $-K_Z=p_1^*(-K_Y)$. Moreover, we have $-K_X=g^*(-K_Y)+\tfrac{1}{2}C$ and $-K_W=G^*(-K_Z)+\tfrac{1}{2}(C\times B)$. Since $-K_Y$ is nef, so is $-K_Z$.
On the other hand, for every $b\in B$, we have $(-K_W)\cdot(C\times\{b\})=(-K_X)\cdot C=-2<0$ which implies that $-K_W$ is not nef. Thus, $G$ is a nontrivial $-K_W$-negative birational contraction, and $Z$ is a $-K_W$-minimal model. Since $Y$ is klt and $B$ is smooth, the product $Z=Y\times B$ is klt, and \cite[Corollary 3.12]{CP16} then implies that $(W,0)$ is potentially klt.

It remains to show that $W$ is not of Calabi--Yau type.
By the K\"{u}nneth formula, we have 
$$H^0(W,-mK_W)=H^0(X,-mK_X)\otimes H^0(B,\mathcal{O}_B)\cong H^0(X,-mK_X),$$ and by the previous example, we have $|-mK_W|=\{p_1^*(mD_X)\}$ for all $m>0$, where $D_X=C+E$. Hence, any effective $\mathbb{Q}$-divisor $\Delta\sim_{\mathbb{Q}}-K_W$ must equal $p_1^*D_X$. Near a point of the form $(q,b)$ with $q\in C\cap E$, the pair $(W,p_1^*D_X)$ is analytically the product of the non-log-canonical surface germ $(X,D_X)_q$ with a smooth curve. Hence, $(W,\Delta)$ is not log canonical. Therefore, $W$ is not of Calabi--Yau type.
\end{example}

\bibliographystyle{habbvr}
\bibliography{biblio}

@article {Sak84,
    AUTHOR = {Sakai, Fumio},
     TITLE = {Anticanonical models of rational surfaces},
   JOURNAL = {Math. Ann.},
  FJOURNAL = {Mathematische Annalen},
    VOLUME = {269},
      YEAR = {1984},
    NUMBER = {3},
     PAGES = {389--410},
      ISSN = {0025-5831,1432-1807},
   MRCLASS = {14J26 (14J17)},
  MRNUMBER = {761313},
MRREVIEWER = {L.\ B\u adescu},
       DOI = {10.1007/BF01450701},
       URL = {https://doi.org/10.1007/BF01450701},
}

@article {HP,
    AUTHOR = {Hwang, DongSeon and Park, Jinhyung},
     TITLE = {Redundant blow-ups of rational surfaces with big anticanonical
              divisor},
   JOURNAL = {J. Pure Appl. Algebra},
  FJOURNAL = {Journal of Pure and Applied Algebra},
    VOLUME = {219},
      YEAR = {2015},
    NUMBER = {12},
     PAGES = {5314--5329},
      ISSN = {0022-4049,1873-1376},
   MRCLASS = {14J26 (14E05 14J17)},
  MRNUMBER = {3390023},
MRREVIEWER = {Amanda\ Knecht},
       DOI = {10.1016/j.jpaa.2015.05.015},
       URL = {https://doi.org/10.1016/j.jpaa.2015.05.015},
}

@article {Leh14,
    AUTHOR = {Lehmann, Brian},
     TITLE = {Algebraic bounds on analytic multiplier ideals},
   JOURNAL = {Annales de l'Institut Fourier},
    VOLUME = {64},
      YEAR = {2014},
    NUMBER = {3},
     PAGES = {1077--1108},
       DOI = {10.5802/aif.2874},
       URL = {https://doi.org/10.5802/aif.2874},
  MRNUMBER = {3330164},
}

@article {LX18,
    AUTHOR = {Li, Chi and Xu, Chenyang},
     TITLE = {Stability of valuations: higher rational rank},
   JOURNAL = {Peking Math. J.},
  FJOURNAL = {Peking Mathematical Journal},
    VOLUME = {1},
      YEAR = {2018},
    NUMBER = {1},
     PAGES = {1--79},
      ISSN = {2096-6075,2524-7182},
   MRCLASS = {14E30 (14B05 32Q26)},
  MRNUMBER = {4059992},
MRREVIEWER = {Jingjun\ Han},
       DOI = {10.1007/s42543-018-0001-7},
       URL = {https://doi.org/10.1007/s42543-018-0001-7},
}

@article {Oht22,
    AUTHOR = {Ohta, Rikito},
     TITLE = {On the relative version of {M}ori dream spaces},
   JOURNAL = {Eur. J. Math.},
  FJOURNAL = {European Journal of Mathematics},
    VOLUME = {8},
      YEAR = {2022},
     PAGES = {S147--S181},
      ISSN = {2199-675X,2199-6768},
   MRCLASS = {14E30 (14E05 14J45)},
  MRNUMBER = {4452841},
MRREVIEWER = {Zhan\ Li},
       DOI = {10.1007/s40879-022-00552-6},
       URL = {https://doi.org/10.1007/s40879-022-00552-6},
}

@article {JM12,
    AUTHOR = {Jonsson, Mattias and Musta\c{t}\u{a}, Mircea},
     TITLE = {Valuations and asymptotic invariants for sequences of ideals},
   JOURNAL = {Ann. Inst. Fourier (Grenoble)},
  FJOURNAL = {Universit\'e{} de Grenoble. Annales de l'Institut Fourier},
    VOLUME = {62},
      YEAR = {2012},
    NUMBER = {6},
     PAGES = {2145--2209},
      ISSN = {0373-0956,1777-5310},
   MRCLASS = {14F18 (12J20 14B05)},
  MRNUMBER = {3060755},
MRREVIEWER = {Carlos\ Galindo},
       DOI = {10.5802/aif.2746},
       URL = {https://doi.org/10.5802/aif.2746},
}

@article {Xu20,
    AUTHOR = {Xu, Chenyang},
     TITLE = {A minimizing valuation is quasi-monomial},
   JOURNAL = {Ann. of Math. (2)},
  FJOURNAL = {Annals of Mathematics. Second Series},
    VOLUME = {191},
      YEAR = {2020},
    NUMBER = {3},
     PAGES = {1003--1030},
      ISSN = {0003-486X,1939-8980},
   MRCLASS = {14E30 (14J17 14J45)},
  MRNUMBER = {4088355},
MRREVIEWER = {Yuchen\ Liu},
       DOI = {10.4007/annals.2020.191.3.6},
       URL = {https://doi.org/10.4007/annals.2020.191.3.6},
}

@book {Xu25,
    AUTHOR = {Xu, Chenyang},
     TITLE = {K-stability of {F}ano varieties},
    SERIES = {New Mathematical Monographs},
    VOLUME = {50},
 PUBLISHER = {Cambridge University Press, Cambridge},
      YEAR = {2025},
     PAGES = {xi+411},
      ISBN = {978-1-009-53877-0; [9781009538763]},
   MRCLASS = {14J45 (32Q26)},
  MRNUMBER = {4893062},
}

@article {CP16,
    AUTHOR = {S. Choi and J. Park},
     TITLE = {Potentially non-klt locus and its applications},
   JOURNAL = {Math. Ann.},
  FJOURNAL = {Mathematische Annalen},
    VOLUME = {366},
      YEAR = {2016},
    NUMBER = {1-2},
     PAGES = {141--166},
    note     = {See \href{https://arxiv.org/abs/1412.8024}  {arXiv:1412.8024v2} for updates.},
      ISSN = {0025-5831,1432-1807},
   MRCLASS = {14E30 (14J17 14J45 14M22)},
  MRNUMBER = {3552236},
MRREVIEWER = {Wenhao\ Ou},
       DOI = {10.1007/s00208-015-1317-6},
       URL = {https://doi.org/10.1007/s00208-015-1317-6},
}

@Misc{S,
  author       = {Shokurov, V. V.},
  howpublished = {\href{https://arxiv.org/abs/2012.06495}{arXiv:2012.06495}},
  title        = {Existence and boundedness of $n$-complements},
  year         = {2020},
  copyright    = {arXiv.org perpetual, non-exclusive license},
  publisher    = {arXiv},
}

@Article{B,
  author    = {C. Birkar},
  journal   = {Ann. of Math. (2)},
  title     = {{Anti-pluricanonical systems on Fano varieties}},
  year      = {2019},
  number    = {2},
  pages     = {345--463},
  volume    = {190},
  doi       = {10.4007/annals.2019.190.2.1},
  fjournal  = {Annals of Mathematics},
  publisher = {Department of Mathematics of Princeton University},
  url       = {https://doi.org/10.4007/annals.2019.190.2.1},
}

@article{G,
	author = {Y. Gongyo},
	Journal = {J. Algebraic Geom.},
	Fjournal = {Journal of Algebraic Geometry},
	pages = {549--564},
	title = {Abundance theorem for numerically trivial log canonical divisors of semi-log canonical pairs},
	volume = {22},
	year = {2012}
}

@Article{Z,
  author  = {D.-Q. Zhang},
  journal = {J. Math. Kyoto. Univ.},
  title   = {Logarithmic {E}nriques surfaces},
  year    = {1991},
  issn    = {2156-2261},
  volume  = {31},
  doi     = {10.1215/kjm/1250519795},
}

@InCollection{Totaro2012,
  author    = {Totaro, Burt},
  booktitle = {Current developments in algebraic geometry. Selected papers based on the presentations at the workshop ``Classical algebraic geometry today'', MSRI, Berkeley, CA, USA, January 26--30, 2009},
  publisher = {Cambridge: Cambridge University Press},
  title     = {Algebraic surfaces and hyperbolic geometry},
  year      = {2012},
  isbn      = {978-0-521-76825-2},
  pages     = {405--426},
  language  = {English},
  zbl       = {1253.14040},
  zbmath    = {6092086},
}

@book {Nak04,
    AUTHOR = {N. Nakayama},
     TITLE = {Zariski-decomposition and abundance},
    SERIES = {MSJ Memoirs},
    VOLUME = {14},
 PUBLISHER = {Mathematical Society of Japan, Tokyo},
      YEAR = {2004},
     PAGES = {xiv+277},
      ISBN = {4-931469-31-0},
   MRCLASS = {14C20 (14E15 14E30 14J10)},
  MRNUMBER = {2104208},
MRREVIEWER = {Tommaso\ De Fernex},
}

@article{SaitoTakebeTerajima,
  author       = {Saito, Masa-Hiko and Takebe, Taro and Terajima, Hitomi},
  title        = {Deformation of {O}kamoto--{P}ainlev\'e pairs and {P}ainlev\'e equations},
  journal      = {Journal of Algebraic Geometry},
  volume       = {11},
  number       = {2},
  pages        = {311--362},
  year         = {2002},
  doi          = {10.1090/S1056-3911-01-00316-2},
  eprint       = {math/0006026},
  archivePrefix= {arXiv},
  primaryClass = {math.AG}
}

@article {LX25,
    AUTHOR = {Liu, Jihao and Xie, Lingyao},
     TITLE = {Relative {N}akayama-{Z}ariski decomposition and minimal models
              of generalized pairs},
   JOURNAL = {Peking Math. J.},
  FJOURNAL = {Peking Mathematical Journal},
    VOLUME = {8},
      YEAR = {2025},
    NUMBER = {2},
     PAGES = {299--349},
      ISSN = {2096-6075,2524-7182},
   MRCLASS = {14E30 (14C20 14E05 14J17)},
  MRNUMBER = {4910986},
MRREVIEWER = {Sheng\ Meng},
       DOI = {10.1007/s42543-023-00076-2},
       URL = {https://doi.org/10.1007/s42543-023-00076-2},
}

@article{Favre,
  author       = {Favre, Charles},
  title        = {Holomorphic self-maps of singular rational surfaces},
  journal      = {Publicacions Matem\`atiques},
  volume       = {54},
  number       = {2},
  pages        = {389--432},
  year         = {2010},
  doi          = {10.5565/PUBLMAT_54210_06},
  eprint       = {0809.1724},
  archivePrefix= {arXiv},
  primaryClass = {math.AG}
}

@article {CHP,
    AUTHOR = {Choi, SungRak and Hwang, DongSeon and Park, Jinhyung},
     TITLE = {Factorization of anticanonical maps of {F}ano type variety},
   JOURNAL = {Int. Math. Res. Not. IMRN},
  FJOURNAL = {International Mathematics Research Notices. IMRN},
      YEAR = {2015},
    NUMBER = {20},
     PAGES = {10118--10142},
      ISSN = {1073-7928,1687-0247},
   MRCLASS = {14J45 (14E30)},
  MRNUMBER = {3455861},
MRREVIEWER = {Sho\ Tanimoto},
       DOI = {10.1093/imrn/rnu274},
       URL = {https://doi.org/10.1093/imrn/rnu274},
}

@article {CJL25,
    AUTHOR = {Choi, SungRak and Jang, Sungwook and Lee, Dae-Won},
     TITLE = {On minimal model program and {Z}ariski decomposition of
              potential triples},
   JOURNAL = {Taiwanese J. Math.},
  FJOURNAL = {Taiwanese Journal of Mathematics},
    VOLUME = {29},
      YEAR = {2025},
    NUMBER = {6},
     PAGES = {1261--1274},
      ISSN = {1027-5487,2224-6851},
   MRCLASS = {14E30 (14J17)},
  MRNUMBER = {5003252},
       DOI = {10.11650/tjm/250406},
       URL = {https://doi.org/10.11650/tjm/250406},
}

@article {Nik,
    AUTHOR = {Nikulin, Viacheslav V.},
     TITLE = {A remark on algebraic surfaces with polyhedral {M}ori cone},
   JOURNAL = {Nagoya Math. J.},
  FJOURNAL = {Nagoya Mathematical Journal},
    VOLUME = {157},
      YEAR = {2000},
     PAGES = {73--92},
      ISSN = {0027-7630,2152-6842},
   MRCLASS = {14J26 (14C22)},
  MRNUMBER = {1752476},
MRREVIEWER = {Arnaud\ Beauville},
       DOI = {10.1017/S0027763000007194},
       URL = {https://doi.org/10.1017/S0027763000007194},
}

@misc{KL,
      title={Minimal model program on the generic fiber of log {C}alabi-{Y}au type fibration}, 
      author={Donghyeon Kim and Dae-Won Lee},
      year={2025},
      howpublished = {\href{https://arxiv.org/abs/2512.02429}{arXiv:2512.02429}},
}

@Article{Okawa2016,
  author   = {S. Okawa},
  journal  = {Math. Ann.},
  title    = {On images of {Mori} dream spaces},
  year     = {2016},
  issn     = {0025-5831},
  number   = {3-4},
  pages    = {1315--1342},
  volume   = {364},
  doi      = {10.1007/s00208-015-1245-5},
  language = {English},
  zbl      = {1341.14007},
  zbmath   = {6559840},
}

@book {KM98,
    AUTHOR = {Koll\'ar, J\'anos and Mori, Shigefumi},
     TITLE = {Birational geometry of algebraic varieties},
    SERIES = {Cambridge Tracts in Mathematics},
    VOLUME = {134},
      NOTE = {With the collaboration of C. H. Clemens and A. Corti,
              Translated from the 1998 Japanese original},
 PUBLISHER = {Cambridge University Press, Cambridge},
      YEAR = {1998},
     PAGES = {viii+254},
      ISBN = {0-521-63277-3},
   MRCLASS = {14E30},
  MRNUMBER = {1658959},
MRREVIEWER = {Mark\ Gross},
       DOI = {10.1017/CBO9780511662560},
       URL = {https://doi.org/10.1017/CBO9780511662560},
}

@book {Kol13,
    AUTHOR = {Koll\'ar, J\'anos},
     TITLE = {Singularities of the minimal model program},
    SERIES = {Cambridge Tracts in Mathematics},
    VOLUME = {200},
      NOTE = {With a collaboration of S\'andor Kov\'acs},
 PUBLISHER = {Cambridge University Press, Cambridge},
      YEAR = {2013},
     PAGES = {x+370},
      ISBN = {978-1-107-03534-8},
   MRCLASS = {14E30 (14B05)},
  MRNUMBER = {3057950},
MRREVIEWER = {Tommaso\ De Fernex},
       DOI = {10.1017/CBO9781139547895},
       URL = {https://doi.org/10.1017/CBO9781139547895},
}

@Misc{CJKL25,
  author       = {S. Choi and S. Jang and D. Kim and D.-W. Lee},
  howpublished = {\href{https://arxiv.org/abs/2506.13637}{arXiv:2506.13637}},
  title        = {A valuative approach to the anticanonical minimal model program},
  year         = {2025},
}

\end{document}